\documentclass[a4paper,12pt]{amsart}
\usepackage{amsmath}
\usepackage{amsmath}
\usepackage{amssymb}
\usepackage{amscd}
\usepackage{mathrsfs}
\usepackage[all]{xy}
\usepackage[pdftex]{graphicx}
\def\mathcal{\mathscr}
\everymath{\displaystyle}
\newtheorem{thm}{Theorem}[section]
\newtheorem{lem}[thm]{Lemma}

\newtheorem{prop}[thm]{Proposition}
\newtheorem{conj}[thm]{Conjecture}
\theoremstyle{definition}

\newtheorem{rem}[thm]{Remark}

\newtheorem{defn}[thm]{Definition}

\newtheorem{que}[thm]{Question}

\newcommand{\mca}[1]{{\mathcal{#1}}}

\def\Q{{\mathbb Q}}

\def\Z{{\mathbb Z}}
\def\R{{\mathbb R}}

\def\AS{\text{\rm AS}}

\def\cont{\text{\rm Cont}}

\def\deg{\text{\rm deg}\,}
\def\del{\partial}

\def\dim{\text{\rm dim}\,}

\def\dR{\text{\rm dR}\,} 
\def\ep{\varepsilon} 

\def\CH{\text{\rm CH}}
\def\ecc{\text{\rm ECC}}
\def\ech{\text{\rm ECH}}

\def\fvect{\text{\rm FVect}}

\def\Hom{\text{\rm Hom}} 
\def\id{\text{\rm id}}
\def\image{\text{\rm Im}\,}

\def\ph{\varphi}

\def\spann{\text{\rm span}\,} 
\def\spec{\text{\rm Spec}\,} 
\def\supp{\text{\rm supp}\,}

\begin{document}
\pagestyle{plain} 
\thispagestyle{plain} 

\title[Strong closing property of contact forms and action selecting functors]
{Strong closing property of contact forms and action selecting functors}

\author[Kei Irie]{Kei Irie}
\address{Research Institute for Mathematical Sciences, Kyoto University, Kyoto 606-8502, JAPAN} 
\email{iriek@kurims.kyoto-u.ac.jp} 

\begin{abstract}
We introduce a notion of strong closing property of contact forms, inspired by the $C^\infty$ closing lemma for Reeb flows in dimension three. 
We then prove a sufficient criterion for strong closing property, which is formulated by considering a monoidal functor from a category of manifolds with contact forms to a category of filtered vector spaces. 
As a potential application of this criterion, we propose a conjecture which says that a standard contact form on the boundary of any symplectic ellipsoid satisfies strong closing property. 
\end{abstract} 

\maketitle

\section{Introduction} 

\subsection{Setup and notations} 

Throughout this paper, 
unless otherwise specified, 
all manifolds, bundles and maps between them are of $C^\infty$. 

Let $n \in \Z_{\ge 1}$ and $Y$ be a $2n-1$-dimensional manifold. 
A \textit{contact form} on $Y$ is $\lambda \in \Omega^1(Y)$ 
such that $\lambda \wedge d\lambda^{n-1}(y) \ne 0$ for any $y \in Y$. 
The \textit{Reeb vector field} $R_\lambda \in \mca{X}(Y)$
is defined by equations 
$i_{R_\lambda} d\lambda \equiv 0$
and $\lambda(R_\lambda) \equiv 1$. 

Let $S^1:= \R/\Z$. 
A \textit{periodic Reeb orbit} of $\lambda$ is a map $\gamma: S^1 \to Y$ such that there exists $T_\gamma \in \R_{>0}$ satisfying 
$\dot{\gamma} \equiv T_\gamma \cdot R_\lambda(\gamma)$.   
The set of all periodic Reeb orbits of $\lambda$ is denoted by $P(Y, \lambda)$. 
For any contact form $\lambda$ on $Y$, 
$\xi_\lambda:= \ker \lambda$ is a hyperplane field on $Y$. 
The hyperplane field $\xi_\lambda$ is co-oriented, i.e. 
$TY/\xi_\lambda$ is oriented in the following way: 
for any $v \in TY$, $[v]$ is positive if $\lambda(v)>0$. 

A \textit{contact manifold} is a pair $(Y, \xi)$ such that $\xi$ is a co-oriented hyperplane field on $Y$ 
and there exists a contact from $\lambda$ on $Y$ satisfying $\xi_\lambda = \xi$. 
For any contact manifold $(Y, \xi)$, let 
$\Lambda(Y, \xi):= \{ \lambda \in \Omega^1(Y) \mid \text{ $\xi_\lambda=\xi$ as a co-oriented hyperplane field} \}$.
The map $C^\infty(Y) \to \Lambda(Y, \xi); \, h \mapsto e^h \lambda$ is a bijection for any $\lambda \in \Lambda(Y, \xi)$. 

\begin{rem}\label{rem_order} 
Let us define a partial order $\le$ on $\Lambda(Y,\xi)$ as follows: 
$\lambda \le \lambda'$ if there exists $h \in C^\infty(Y, \R_{\ge 0})$ 
which satisfies $e^h \lambda = \lambda'$. 
Then, $\Lambda(Y, \xi)$ with this partial order $\le$ (as well as its inverse) 
is a directed set. 
\end{rem} 

\subsection{Strong closing property} 

Let us introduce the key definition in this paper. 

\begin{defn}\label{defn_SCP} 
Let $(Y, \xi)$ be a closed contact manifold. 
Let us say that $\lambda \in \Lambda(Y, \xi)$ satisfies 
\textit{positive (resp. negative) strong closing property} if the following holds: 
\begin{quote} 
For any $h \in C^\infty(Y, \R_{\ge 0})$ (resp. $C^\infty(Y, \R_{\le 0})$)
with $h \not\equiv 0$, 
there exist $t \in [0,1]$ and $\gamma \in P(Y, e^{th} \lambda)$
such that $\image \gamma  \cap \supp h \ne \emptyset$. 
\end{quote} 
Let us say that $\lambda$ satisfies \textit{strong closing property} if it satisfies 
positive or negative strong closing property. 
\end{defn} 

\begin{rem} 
If $\lambda$ satisfies strong closing property, 
then $\lambda$ satisfies the following property (which we call local $C^\infty$ closing property):  
\begin{quote} 
If $U$ is any nonempty open set of $Y$ and $\mca{U}$ is any neighborhood of $0$ in $C^\infty(Y)$ with the $C^\infty$ topology, 
then there exist $h \in \mca{U}$ and $\gamma \in P(Y, e^h \lambda)$ such that 
$\supp h \subset U$ and $\image \gamma \cap U \ne \emptyset$. 
\end{quote} 
\end{rem} 

\begin{rem} 
If $\bigcup_{\gamma \in P(Y, \lambda)} \image \gamma $ is dense in $Y$, then $\lambda$ satisfies strong closing property. 
\end{rem} 

\subsection{Results in low dimensions} 

The following result is implicit in the proof of \cite{Irie_dense} Lemma 3.1. 

\begin{thm}[\cite{Irie_dense}]\label{thm_SCP_dim3} 
Any contact form on any closed three-manifold satisfies strong closing property. 
\end{thm} 

The proof in \cite{Irie_dense} used the volume theorem (or Weyl law) for action selectors (or spectral invariants) 
defined from embedded contact homology (ECH), which was proved by \cite{CGHR}. 
Theorem \ref{thm_SCP_dim3} implies density of periodic Reeb orbits 
for a $C^\infty$ generic contact form on any closed contact three-manifold (\cite{Irie_dense} Theorem 1.1). 
This result does not extend to general Hamiltonian systems even in dimension four, 
due to a counterexample by \cite{Herman}. 
On the other hand, the Hamiltonian $C^1$ closing lemma \cite{Pugh_Robinson} 
implies that on any closed contact manifold (of any dimension) equipped with a $C^2$ generic contact form, 
the union of periodic Reeb orbits is dense. 

There are also similar results for surface maps. 
\cite{Asaoka_Irie} proved, as an application of \cite{Irie_dense}, 
a $C^\infty$ closing lemma for Hamiltonian diffeomorphisms of closed symplectic surfaces. 
Recently, \cite{CGPZ} proved an analogue of the volume theorem for spectral invariants defined from periodic Floer homology (PFH), 
and deduced a $C^\infty$ closing lemma for area-preserving diffeomorphisms of closed symplectic surfaces. 
\cite{Edtmair_Hutchings} independently proved an analogue of the volume theorem for PFH spectral invariants 
and $C^\infty$ closing lemmas for area-preserving diffeomorphisms under certain conditions (rationality and $U$-cycle property of Hamiltonian isotopy classes). In particular, \cite{Edtmair_Hutchings} proved a $C^\infty$ closing lemma for area-preserving diffeomorphisms of the symplectic two-torus. Combined with a result in \cite{CGPPZ}, this proof in \cite{Edtmair_Hutchings} extends to closed symplectic surfaces of arbitrary genus. \cite{Edtmair_Hutchings} also proved closing lemmas with quantitative estimates of periods. 

Since ECH and PFH are technologies specific to low dimensions, 
it is highly nontrivial whether the results described above extend to higher dimensions. 
Actually, even the following modest question seems quite nontrivial for current technologies. 

\begin{que}\label{que} 
Does there exist a pair $(Y, \lambda)$ 
such that $Y$ is a closed manifold with $\dim Y \ge 5$, 
$\lambda$ is a contact form on $Y$ which satisfies strong closing property, 
and $\bigcup_{\gamma \in P(Y, \lambda)} \image \gamma$ is not dense in $Y$? 
\end{que} 

\begin{rem} 
Recent papers \cite{CDPT} and \cite{CS} proposed proofs of Conjecture \ref{conj_ellipse}, 
which implies the affirmative answer to Question \ref{que}. 
See also \cite{Xue}, in which a variant of strong closing property for positive definite quadratic Hamiltonians (\cite{Xue} Theorem 1.4) 
is proved via KAM normal form approach. 
\end{rem} 

\subsection{Contents of this paper} 

Motivated by Question \ref{que}, 
and inspired by a perturbation argument in \cite{Irie_dense} and
the proof of the lower bound of the volume theorem for ECH spectral invariants (\cite{CGHR} Section 3), 
we prove a sufficient criterion for strong closing property (Theorem \ref{mainthm}). 
As a potential application of this criterion, 
we propose a conjecture (Conjecture \ref{conj_ellipse}) 
which says that a standard contact form on the boundary of any symplectic ellipsoid satisfies strong closing property.  

\begin{rem} 
\cite{Fish_Hofer} suggests another possible way to prove a $C^\infty$ closing property 
for Hamiltonian dynamics in higher dimensions. 
Namely, it is conjectured (\cite{Fish_Hofer} Conjecture 1.12) that, for any integer $n \ge 2$, 
the $2n$-dimensional symplectic vector space satisfies a $C^\infty$ closing property in the sense of  
\cite{Fish_Hofer} Definition 1.11. 
In \cite{Fish_Hofer} Section 1.4, an idea to attack this conjecture via the theory of feral curves is also discussed. 
\end{rem}

This paper is organized as follows. 
In Section 2, we introduce a notion of contact action selector,  and recall the perturbation argument in \cite{Irie_dense}. 
In Section 3, we introduce a notion of action selecting functor 
as a symmetric (strong) monoidal functor from a category of manifolds with contact forms to a category of filtered vector spaces, 
requiring certain conditions. 
We also explain two examples of action selecting functors defined from ECH and (full) contact homology. 
In Section 4, we state and prove Theorem \ref{mainthm} using the notion of action selecting functor. 
In Section 5, we state Conjecture \ref{conj_ellipse} and very briefly discuss a recently proposed proof \cite{CDPT} of this conjecture. 
In Appendix A, we prove Lemma \ref{lem_filtered_chain_complex}. 
The proof given in the appendix is a straightforward generalization of the argument in \cite{Hutchings_quantitative} Section 5 in a purely algebraic setting. 

{\bf Acknowledgement.} 
I thank Suguru Ishikawa for helpful conversations, 
Helmut Hofer for informing me of \cite{Fish_Hofer}, 
and Julian Chaidez for answering my questions on \cite{CDPT}. 
I also thank the referee for comments which helped me improve the readability of this paper. 
This research is supported by JSPS KAKENHI Grant Number 18K13407 and 19H00636. 

\section{Contact action selectors} 

Let $Y$ be a closed manifold, and $\lambda$ be a contact form on $Y$.
Let $\spec(Y, \lambda)$ denote the set of periods of periodic Reeb orbits of $(Y, \lambda)$: 
\[ 
\spec(Y, \lambda):= \{ T_\gamma \mid \gamma \in P(Y, \lambda) \} \subset \R_{>0}. 
\] 
We also consider the set of finite sums of elements of $\spec(Y, \lambda)$: 
\[ 
\spec_+(Y, \lambda):= \{0\} \cup \bigg\{ \sum_{i=1}^k T_{\gamma_i} \bigg{|}  k \ge 1,  \, \gamma_1, \ldots, \gamma_k \in P(Y, \lambda) \bigg\} \subset \R_{\ge 0}. 
\] 
Note that $\gamma_1, \ldots, \gamma_k$ on the RHS are not assumed to be distinct. 

\begin{lem}\label{lem_spec_+} 
For any $Y$ and $\lambda$ as above, 
$\spec_+(Y, \lambda)$ is a closed set of $\R$. 
It is also a null set with respect to the Lebesgue measure on $\R$. 
\end{lem} 

The proof of Lemma \ref{lem_spec_+} is omitted. 
The proof of \cite{Irie_dense} Lemma 2.2, which is equivalent to Lemma \ref{lem_spec_+} in the case $\dim Y=3$, 
works in arbitrary dimension. 

Let us introduce a notion of contact action selector. 

\begin{defn}\label{defn_action_selector} 
Let $(Y, \xi)$ be a closed contact manifold. 
A \textit{contact action selector} of $(Y, \xi)$ is a map 
$c: \Lambda(Y, \xi) \to \R_{\ge 0}$ 
which satisfies the following properties for any $\lambda \in \Lambda(Y, \xi)$: 
\begin{enumerate} 
\item[(i):] $c(\lambda) \in \spec_+(Y, \lambda)$. 
\item[(ii):] $c(s\lambda) = s c (\lambda)$ for any $s \in \R_{>0}$. 
\item[(iii):] $c(\lambda') \ge c(\lambda)$ for any $\lambda' \in \Lambda(Y, \xi)$ such that $\lambda' \ge \lambda$ (see Remark \ref{rem_order}). 
\end{enumerate} 
Let $\AS(Y, \xi)$ denote the set of contact action selectors of $(Y, \xi)$. 
\end{defn} 

The properties (i), (ii), and (iii) are called spectrality, conformality and monotonicity, respectively. 
Any contact action selector $c$ satisfies the following property: 
\begin{itemize} 
\item[(iv):] If a sequence $(h_j)_{j \ge 1}$ in $C^\infty(Y) $ satisfies $\lim_{j \to \infty} \| h_j \|_{C^0}=0$, 
then $\lim_{j \to \infty} c ( e^{h_j} \lambda) = c(\lambda)$. 
\end{itemize} 
Indeed, by conformality and monotonicity, 
$e^{\min h_j} c(\lambda) \le c(e^{h_j}  \lambda) \le e^{\max h_j} c(\lambda)$ for any $j$. 
The property (iv) is called $C^0$ continuity. 

The following lemma is essentially equivalent to the perturbation argument in \cite{Irie_dense}. 

\begin{lem}\label{lem_local_sensitivity} 
Let $(Y, \xi)$ be a closed contact manifold and $\lambda \in \Lambda(Y, \xi)$. 
\begin{itemize} 
\item[(i):]  Suppose that, for any $h \in C^\infty(Y, \R_{\ge 0}) \setminus \{0\}$, 
there exists $c \in \AS(Y, \xi)$ which satisfies 
$c(e^h \lambda) > c(\lambda)$. 
Then $\lambda$ satisfies positive strong closing property. 
\item[(ii):]  Suppose that, for any $h \in C^\infty(Y, \R_{\le 0}) \setminus \{0\}$, 
there exists $c \in \AS(Y, \xi)$ which satisfies 
$c(e^h \lambda) < c(\lambda)$. 
Then $\lambda$ satisfies negative strong closing property. 
\end{itemize} 
\end{lem} 
\begin{proof} 
We only prove (i), since the proof of (ii) is similar to that. 
First note that (i) is equivalent to the following: 
if $\lambda$ does not satisfy positive strong closing property, 
then there exists $h \in C^\infty(Y, \R_{\ge 0}) \setminus \{0\}$ 
such that $c(e^h\lambda )= c(\lambda)$ for any $c \in \AS(Y, \xi)$. 

If $\lambda$ does not satisfy positive strong closing property, there exists $h \in C^\infty(Y, \R_{\ge 0}) \setminus \{0\}$
such that the following holds: 
$t \in [0,1], \, \gamma \in P(Y, e^{th} \lambda) \implies \image \gamma \cap \supp h  = \emptyset$. 

Let $\lambda_t:= e^{th} \lambda$. 
Since $\lambda_t \equiv \lambda$ on $Y \setminus \supp h$, 
we obtain 
$P(Y, \lambda_t) = P(Y, \lambda)$ and 
$\spec_+(Y, \lambda_t) = \spec_+(Y, \lambda)$ for any $t \in [0,1]$. 
Let $c \in \AS(Y, \xi)$. 
By spectrality, 
$c(\lambda_t) \in \spec_+(Y, \lambda_t) = \spec_+(Y, \lambda)$ for any $t$. 
By $C^0$ continuity, $c(\lambda_t)$ is continuous on $t$. 
Since $\spec_+(Y,\lambda)$ is a null set, 
$c(\lambda_t)$ is constant on $t$, in particular $c(e^h\lambda) = c(\lambda)$. 
\end{proof}

\section{Action selecting functors} 

In this section, we introduce a notion of \textit{action selecting functor} (of dimension $2n-1$ and coefficients in $K$) 
as a symmetric (strong) monoidal functor $\cont_{2n-1} \to \fvect^{\Z/2}_K$, 
requiring certain conditions. 

In Section 3.1, we explain the definition of $\cont_{2n-1}$, 
which is a category whose objects are closed $2n-1$-dimensional manifolds with contact forms. 
In Section 3.2, we explain the definition of $\fvect^{\Z/2}_K$, 
which is a category whose objects are $\Z/2$-graded and filtered $K$-vector spaces. 
In Section 3.3, we explain the definition of an action selecting functor. 
In Section 3.4, we explain the definition of a contact action selector associated to any action selecting functor. 
In Section 3.6, we introduce a subcategory $\cont^0_{2n-1} \subset \cont_{2n-1}$ to prepare for 
Section 3.7, in which we explain two examples of action selecting functors $\Phi^{\ech}$ and $\Phi^{\CH}$, 
which are defined respectively by ECH and (full) contact homology. 

\subsection{Category $\cont_{2n-1} $} 

Let $n \in \Z_{\ge 1}$. The goal of this subsection is to define a symmetric nonunital monoidal category $\cont_{2n-1}$. 

An object of $\cont_{2n-1} $ is a pair $(Y, \lambda)$ such that $Y$ is a closed $2n-1$-dimensional manifold 
and $\lambda$ is a contact form on $Y$. 

For any $(Y_+, \lambda_+), (Y_-, \lambda_-) \in \cont_{2n-1} $, 
a generalized exact cobordism from $(Y_+, \lambda_+)$ to $(Y_-, \lambda_-)$
is a tuple 
$(X, \lambda, i_+, i_-)$ which satisfies the following conditions: 
\begin{itemize} 
\item $X$ is a $2n$-dimensional manifold without boundary and $\lambda\in \Omega^1(X)$. 
Moreover, $d\lambda$ is nondegenerate (thus symplectic). 
\item $i_+: Y_+ \times \R  \to X $ is an open embedding such that $i_+^* \lambda = e^r \lambda_+$, where $r$ denotes the coordinate on $\R$. 
\item $i_-: Y_- \times \R \to X$ is an open embedding such that $i_-^*\lambda = e^r \lambda_-$. 
\item $U_+:= i_+(Y_+ \times \R_{>0})$ and $U_-:= i_-(Y_- \times \R_{<0})$ are disjoint, and $X \setminus  (U_+ \cup U_-)$ is compact. 
\end{itemize} 
Generalized exact cobordisms 
$(X, \lambda, i_+, i_-)$ and $(X', \lambda', i'_+, i'_-)$
are equivalent 
if there exists a diffeomorphism $\ph: X \to X'$ such that 
$d(\ph^*\lambda' - \lambda)=0$, 
$\ph \circ i_+ = i'_+$, 
$\ph \circ i_- = i'_-$, 
and 
\[ 
\supp (\ph^*\lambda' - \lambda) \cap ( i_+(Y_+ \times \R_{\ge 0}) \cup i_-(Y_- \times \R_{\le 0})) = \emptyset. 
\] 
We define $\Hom( (Y_+, \lambda_+), (Y_-, \lambda_-))$ 
to be the set of equivalence classes of generalized exact cobordisms.

For any $(Y, \lambda), (Y', \lambda'), (Y'', \lambda'') \in \cont_{2n-1}$, the composition 
\[ 
\Hom ((Y', \lambda'), (Y, \lambda)) \times \Hom ( (Y'', \lambda''), (Y', \lambda')) \to \Hom( (Y'', \lambda''), (Y, \lambda)); \,
(\mu, \mu') \mapsto \mu \circ \mu' 
\] 
is defined as follows. Take $(X, \lambda, i_+, i_-)$ and $(X', \lambda', i'_+, i'_-)$ so that 
\[ 
\mu = [(X, \lambda, i_+, i_-)], \qquad \mu' = [(X', \lambda', i'_+, i'_-)]. 
\] 
Let us consider an equivalence relation $\sim$ on $X \sqcup X'$ generated by 
$i_+(y,r) \sim i'_-(y,r)$, where $(y,r)$ runs over all elements of $Y' \times \R$. 
Let $X'':= (X \sqcup X')/\sim$ and $\pi: X \sqcup X' \to X''$ be the canonical map. 
Let 
$\lambda'':= \pi_! ( \lambda \sqcup \lambda')$, 
$i''_+:= \pi \circ i'_+$, and $i''_-:= \pi \circ i_-$. 
Then we define $\mu \circ \mu':= [( X'', \lambda'', i''_+, i''_-)]$. 

For any $(Y, \lambda), (Y, \lambda') \in \cont_{2n-1} $ such that 
$\lambda \le \lambda'$, i.e. 
there exists $h \in C^\infty(Y, \R_{\ge 0})$ 
which satisfies $\lambda' = e^h \lambda$, 
we define $\mu_{\lambda \lambda'} \in \Hom( (Y, \lambda'), (Y, \lambda))$ by 
\begin{equation}\label{eqn_mu} 
\mu_{\lambda \lambda'}:= [ (Y \times \R, e^r\lambda, i_+, i_-)] 
\end{equation}
where $i_\pm$ are defined as follows: 
\begin{align*} 
i_+&: Y \times \R  \to Y \times \R; \, (y,r) \mapsto (y, h(y) + r ), \\ 
i_-&: Y \times \R \to Y\times \R; \, (y,r) \mapsto (y, r). 
\end{align*} 
For any $(Y, \lambda)\in \cont_{2n-1} $, we define 
$\id_{(Y, \lambda)}:= \mu_{\lambda \lambda}$. 

Now we have defined the category $\cont_{2n-1}$. 
It has a natural symmetric nonunital monoidal structure, 
where the monoidal product is defined by taking disjoint unions.

\subsection{Category $\fvect^{\Z/2}_K$} 

Let $K$ be a field. 
In this subsection, we define a symmetric monoidal category $\fvect^{\Z/2}_K$. 
We also explain a few facts about filtered vector spaces. 

Let $V$ be a vector space over $K$. 
In this paper, a filtration on $V$ means a family $(F^aV)_{a \in (-\infty, \infty]}$ which satisfies the following conditions: 
\begin{itemize} 
\item $F^0 V = \{0\}$, $F^\infty V = V$. 
\item $a \le b \implies F^aV \subset F^bV$. 
\item For any $a \in (-\infty, \infty]$, there holds $\bigcup_{b < a} F^b V = F^a V$. 
\end{itemize} 

When $V$ is $\Z/2$-graded, i.e. equipped with a direct sum decomposition $V = V_0 \oplus V_1$, 
we also require that
$F^a V = (F^a V \cap V_0) \oplus (F^aV \cap V_1)$ holds for any $a$. 

For any filtered vector spaces $V$ and $W$, 
we say that a linear map $f: V \to W$ preserves filtrations if 
$f(F^aV) \subset F^a W$ for any $a$.

For any filtered vector space $V$, let 
$\spec (V):= \{ a \in \R \mid \forall \ep  \in \R_{>0} , \, F^{a-\ep} V \subsetneq F^{a+\ep} V\}$. 
For any $v \in V$, let $|v|:= \inf\{ a \mid v \in F^a V\}$. 
The following properties are easy to confirm: 
\begin{itemize} 
\item $\spec (V)$ is a closed subset of $\R_{\ge 0}$. 
\item If $a \le  b$ and $[a,b] \cap \spec(V) = \emptyset$, then $F^a V= F^b V$. 
\item $|v| \in \spec(V)$ for any $v \in V \setminus \{0\}$.
\item $|0| = -\infty$. 
\end{itemize} 

Let us define the category $\fvect^{\Z/2}_K$. 
Objects of $\fvect^{\Z/2}_K$ are filtered and $\Z/2$-graded $K$-vector spaces, 
and morphisms are degree $0$ linear maps which preserve filtrations. 
The monoidal product is defined by taking tensor products. 
Namely, for any filtered vector spaces $V$ and $W$, 
the filtration on $V \otimes W$ is defined as follows: 
\begin{align*} 
F^a (V \otimes W):&=  \spann \{ F^b V \otimes F^c W \mid b+c \le a \} \\ 
&= \spann \{ v \otimes w \mid v \in V, \, w \in W, \, |v| + |w| < a\}. 
\end{align*} 
Then $\fvect^{\Z/2}_K$ has a natural symmetric monoidal structure
with isomorphisms 
\[ 
V \otimes W \cong W \otimes V; \, v \otimes w \mapsto (-1)^{\deg(v) \deg(w) } w \otimes v, 
\] 
where $\deg$ denotes $\Z/2$-grading. 

\begin{rem}\label{rem_ungraded} 
We also consider a subcategory $\fvect_K \subset \fvect^{\Z/2}_K$ 
which consists of $V \in \fvect^{\Z/2}_K$ such that 
$V_0 = V$ and $V_1 = 0$. 
\end{rem} 

Let us now explain a few facts about filtered vector spaces. 

For any filtered $K$-vector space $V$, let $V^\vee:= \Hom_K(V, K)$. 
For any $\ph \in V^\vee \setminus \{0\}$, let 
$| \ph|:= \sup \{ a \mid  \ph|_{F^aV} = 0 \}$. 
It is easy to see that $|\ph| \in \spec (V)$.

\begin{lem}\label{lem_level_decreasing} 
Let $V$ and $W$ be filtered $K$-vector spaces, 
and let $\ph \in V^\vee$. 
Let us define $i_\ph: V \otimes W \to W$ by 
$i_\ph (v \otimes w) := \ph(v) \cdot w$. 
If $\ph \ne 0$, then 
$|i_\ph(x)| \le |x|- |\ph|$ for any $x \in V \otimes W$. 
\end{lem} 
\begin{proof} 
It is sufficient to show the following claim for any $a,b \in \R$: 
\[ 
|x| < a, \, |\ph| > b \implies |i_\ph(x) | \le a - b. 
\] 
By $|x|<a$, 
one has $x = \sum_i c_i \cdot v_i \otimes w_i$ 
such that $\max_i  (|v_i| + |w_i|)  < a$. 
Then $i_\ph(x) = \sum_i c_i \cdot \ph(v_i) \cdot w_i$. 
By $|\ph|>b$, 
if $\ph(v_i) \ne 0$
then $|v_i| \ge b$, 
which implies $|w_i| < a-b$. 
Thus we obtain $|i_\ph(x)| < a-b$. 
\end{proof} 

A filtered, $\Z/2$-graded chain complex is a pair $(V, \del_V)$ such that 
$V \in \fvect^{\Z/2}_K$
and $\del_V: V_* \to V_{*-1}$ is a linear map 
which preserves the filtration. 
For any such $(V, \del_V)$, its homology $H_*(V)$ has a filtration defined by 
$F^a H_*(V):= \{ [v] \in H_*(V) \mid v \in F^a V\}$ for any $a$, 
thus $H_*(V) \in \fvect^{\Z/2}_K$. 

\begin{lem}\label{lem_filtered_chain_complex} 
Let $(V, \del_V)$ and $(W, \del_W)$ be filtered, $\Z/2$-graded chain complexes. 
Consider the product complex $(V \otimes W, \del_{V \otimes W})$ 
which is defined by 
\[ 
\del_{V \otimes W}(v \otimes w):= \del_V v \otimes w + (-1)^{\deg(v)} v \otimes \del_W w. 
\] 
If $\dim F^a H_*(V), \dim F^a H_*(W) < \infty$ for any $a \in \R$, then 
\[ 
F^a H_*(V \otimes W)  = \spann \{F^b H_*(V) \otimes F^c H_*(W) \mid b+c \le a \}
\]  
for any $a \in (-\infty, \infty]$. 
\end{lem} 

One can prove Lemma \ref{lem_filtered_chain_complex} 
by a straightforward generalization of Step 2 in the proof of \cite{Hutchings_quantitative} Proposition 4.6. 
For the convenience of the reader, we explain a detailed proof in Appendix A.

\subsection{Action selecting functors} 

Let $n \in \Z_{\ge 1}$ and $K$ be a field. 

\begin{defn}\label{defn_action_selecting_functor} 
An \textit{action selecting functor} 
of dimension $2n-1$ and coefficients in $K$ 
is a symmetric strong monoidal functor 
$\Phi: \cont_{2n-1} \to \fvect^{\Z/2}_K$ such that, 
for any $(Y, \lambda) \in \cont_{2n-1}$ 
and $a,b \in \R_{>0}$ 
satisfying $a<b$ and 
$[a,b] \cap \spec_+(Y, \lambda)=\emptyset$, 
the following holds: 
\begin{itemize} 
\item  $F^a\Phi(Y, \lambda) = F^b\Phi(Y, \lambda)$. 
\item Let $c:=a/b$. 
Then, the linear map
$\Phi(\mu_{c\lambda, \lambda})|_{F^a \Phi (Y,\lambda)}: F^a\Phi(Y, \lambda) \to F^a \Phi(Y, c\lambda)$ is an isomorphism. 
See (\ref{eqn_mu}) for the definition of the morphism $\mu_{c\lambda, \lambda}$. 
\end{itemize} 
\end{defn} 

\begin{rem} 
By a ``strong'' monoidal functor, we mean that  the map 
\[ 
\Phi(Y_1, \lambda_1) \otimes \Phi(Y_2, \lambda_2) \to \Phi(Y_1 \sqcup Y_2, \lambda_1 \sqcup \lambda_2)
\] 
is an isomorphism of $\fvect^{\Z/2}_K$ for any $(Y_1, \lambda_1), (Y_2, \lambda_2) \in \cont_{2n-1}$. 
\end{rem} 

Let us examine some consequences of Definition \ref{defn_action_selecting_functor}. 
In the rest of this subsection, $\Phi$ denotes an action selecting functor of dimension $2n-1$ and coefficients in $K$. 

\begin{lem}\label{lem_AS_functor_spec} 
For any $(Y, \lambda) \in \cont_{2n-1}$, 
there holds $\spec \Phi(Y,\lambda) \subset \spec_+(Y,\lambda)$. 
\end{lem} 
\begin{proof} 
Recall that $\spec_+(Y,\lambda)$ is closed (Lemma \ref{lem_spec_+}). 
For any $a \in \R_{>0} \setminus \spec_+(Y,\lambda)$, 
there exists $\ep \in \R_{>0}$ such that 
$[a-\ep, a+\ep] \cap \spec_+(Y, \lambda) = \emptyset$. 
Then $F^{a-\ep}\Phi(Y,\lambda) = F^{a+\ep}\Phi(Y, \lambda)$, 
which implies $a \not\in \spec \Phi(Y,\lambda)$. 
\end{proof} 

For any 
$(Y, \lambda) \in \cont_{2n-1}$, 
$a \in (0, \infty]$ and $s \in (0,1]$, 
Let us define an isomorphism 
\begin{equation}\label{eqn_scaling_isom} 
\ph_{s,a} : F^a \Phi(Y, \lambda)  \to  F^{sa} \Phi (Y, s\lambda) 
\end{equation}
as follows. 
First we consider the case $a \notin \spec_+(Y, \lambda) \cup \{\infty\}$. 
Then we can take a sequence $(s_0, s_1, \ldots, s_N)$ such that 
$1=s_0> s_1 > \cdots > s_N=s$ 
and 
$[a, (s_{i-1}/s_i) a] \cap \spec_+(Y, \lambda) = \emptyset$
for any $i$. 
This condition is equivalent to 
$[s_ia , s_{i-1}a] \cap \spec(Y, s_i \lambda) = \emptyset$, 
thus 
$F^{s_i a} \Phi(Y, s_i \lambda) = F^{s_{i-1} a } \Phi(Y, s_i \lambda)$. 
Then we can define an isomorphism 
\[ 
\ph_i: F^{s_{i-1}a} \Phi (Y, s_{i-1} \lambda) \to F^{s_{i-1}a}(Y, s_i \lambda) =  F^{s_ia}  (Y, s_i \lambda), 
\] 
where the first map is a restriction of $\Phi(\mu_{s_i \lambda, s_{i-1}\lambda})$. 
Then let us  define 
\[ 
\ph_{s,a}:= \ph_N \circ \cdots \circ \ph_1.
\] 
It is easy to check that the RHS does not depend on the choice of $s_1, \ldots, s_N$. 

For any $a \le a'$ such that $a, a' \notin \spec_+(Y, \lambda) \cup \{\infty \}$, 
the diagram 
\begin{equation}\label{diagram_scaling_isom} 
\xymatrix{ 
F^a \Phi(Y, \lambda)  \ar[r]^-{\ph_{s,a}} \ar[d]& F^{sa} \Phi(Y, s\lambda)  \ar[d]\\ 
F^{a'} \Phi(Y, \lambda) \ar[r]_-{\ph_{s,a'}} & F^{sa'} \Phi(Y, s\lambda) 
}
\end{equation}
commutes, where vertical maps are inclusion maps. 

Let us define the map (\ref{eqn_scaling_isom}) for any $a \in (0, \infty]$. 
Take a sequence $a_1  <  a_2 < \cdots < a_i < \cdots$ 
such that $a_i \notin \spec_+(Y, \lambda)$ for any $i$ and $\lim_{i \to \infty} a_i = a$. 
Then let us define $\ph_{s,a}:= \varinjlim_i \ph_{s, a_i}$.  
This completes the definition of (\ref{eqn_scaling_isom}). 

\begin{rem}\label{rem_scaling_invariance} 
It is easy to check that the following holds for any $(Y, \lambda) \in \cont_{2n-1}$: 
\begin{itemize} 
\item $\ph_{s,\infty} = \Phi(\mu_{s\lambda, \lambda})$ for any $s \in (0,1]$. 
In particular, $\Phi(\mu_{s\lambda, \lambda}): \Phi(Y, \lambda) \to \Phi(Y, s\lambda)$ is an isomorphism. 
\item For any $s \in (0,1]$ and $a \le a'$, the diagram (\ref{diagram_scaling_isom}) commutes. 
\end{itemize} 
\end{rem}

\subsection{Contact action selectors defined from action selecting functors} 

Let $\Phi: \cont_{2n-1} \to \fvect^{\Z/2}_K$ be an action selecting functor. 

\begin{lem} 
For any $(Y,\lambda), (Y, \lambda') \in \cont_{2n-1}$
such that $\lambda \le \lambda'$, 
the map $\Phi(\mu_{\lambda \lambda'}): \Phi(Y, \lambda') \to \Phi(Y, \lambda)$ is an isomorphism of vector spaces. 
\end{lem} 
\begin{proof} 
By $\lambda \le \lambda'$, there exists $h \in C^\infty(Y, \R_{\ge 0})$ such that
$\lambda' = e^h \lambda$. 
Take $C \in \R_{>0}$ so that $0 \le h(y) \le \log C$ for any $y \in Y$. 
Note that 
\[ 
\Phi(\mu_{C^{-1}\lambda', \lambda'}) = \Phi(\mu_{C^{-1}\lambda', \lambda}) \circ \Phi(\mu_{\lambda \lambda'}), \qquad
\Phi(\mu_{\lambda, C\lambda}) = \Phi(\mu_{\lambda \lambda'}) \circ \Phi(\mu_{\lambda', C\lambda}). 
\] 
By Remark \ref{rem_scaling_invariance}, 
$\Phi(\mu_{C^{-1}\lambda', \lambda'})$ and $\Phi(\mu_{\lambda, C\lambda})$ are isomorphisms. 
Hence $\Phi(\mu_{\lambda \lambda'})$ is an isomorphism. 
\end{proof} 
\begin{rem} 
$\Phi(\mu_{\lambda \lambda'})$ is not necessarily an isomorphism of filtered vector spaces. 
\end{rem} 

Let $(Y, \xi)$ be a closed contact manifold of dimension $2n-1$. 
Let us define a $\Z/2$-graded vector space 
\[ 
\Phi(Y,\xi):= \varinjlim_{\lambda \in \Lambda(Y,\xi)} \Phi(Y, \lambda) 
\] 
where the direct limit on the RHS is taken over
the directed set $\Lambda(Y, \xi)$ with the \textit{inverse} of $\le$ (see Remark \ref{rem_order}). 

For any $\lambda \in \Lambda(Y, \xi)$, 
let $I_\lambda: \Phi(Y,\lambda) \to \Phi(Y, \xi)$ denote the canonical isomorphism. 
For any $\sigma \in \Phi(Y,\xi)$, let 
\[ 
c^\Phi_\sigma(\lambda):= |I_\lambda^{-1}(\sigma)| = \inf\{ a \mid  \sigma \in I_\lambda(F^a \Phi(Y,\lambda)) \}. 
\] 
Note that $c^\Phi_0(\lambda)= |0| = -\infty$ for any $\lambda$. 

\begin{lem} 
If $\sigma \ne 0$, then 
$c^\Phi_\sigma: \Lambda(Y, \xi) \to \R_{\ge 0}$ is a contact action selector. 
\end{lem}
\begin{proof} 
Let us take $\lambda \in \Lambda(Y, \xi)$ arbitrarily, 
and verify the conditions in Definition \ref{defn_action_selector}, i.e. 
(i) spectrality, (ii) conformality, and (iii) monotonicity. 

(i): since $I_\lambda$ is an isomorphism, $I_\lambda^{-1}(\sigma) \ne 0$. Hence
\[ 
c^\Phi_\sigma(\lambda) = | I_\lambda^{-1}(\sigma)|  \in \spec \Phi(Y,\lambda)  \subset  \spec_+(Y, \lambda). 
\] 

(ii): it is sufficient to show $c_\sigma(s\lambda) = s c_\sigma(\lambda)$ for any $s \in (0, 1]$. 
For any $a \in (0, \infty]$, consider the following diagram: 
\[ 
\xymatrix{ 
F^a \Phi(Y, \lambda) \ar[r]\ar[d]_-{\ph_{s,a}}& \Phi(Y, \lambda)  \ar[r]^-{I_\lambda} \ar[d]_-{\Phi(\mu_{s\lambda,\lambda})}& \Phi(Y, \xi) \\
F^{sa} \Phi(Y, s\lambda) \ar[r] & \Phi (Y, s\lambda). \ar[ru]_{I_{s\lambda}}
} 
\] 
The left square commutes by Remark \ref{rem_scaling_invariance}, 
and the right triangle commutes by the definition of $I_\lambda$ and $I_{s\lambda}$. 
Hence we obtain $I_\lambda(F^a\Phi(Y,\lambda)) = I_{s\lambda}(F^{sa}\Phi(Y,s\lambda))$ for any $a$, 
which implies that $c_\sigma(s\lambda) = s c_\sigma(\lambda)$. 

(iii): for any $\lambda' \in \Lambda(Y, \xi)$ which satisfies $\lambda' \ge \lambda$, 
the map $\Phi(\mu_{\lambda\lambda'}): \Phi(Y, \lambda') \to \Phi(Y, \lambda)$ 
preserves filtrations 
(because, by definition, every morphism of $\fvect^{\Z/2}_K$ preserves filtrations) 
and sends $I_{\lambda'}^{-1}(\sigma)$ to $I_\lambda^{-1}(\sigma)$. 
Thus $c^\Phi_\sigma(\lambda) = |I_\lambda^{-1}(\sigma)| \le |I_{\lambda'}^{-1}(\sigma)| = c^\Phi_\sigma(\lambda')$. 
\end{proof} 

\subsection{Category $\cont^0_{2n-1}$}

For any $n \in \Z_{\ge 1}$, we define a subcategory $\cont^0_{2n-1} \subset \cont_{2n-1}$ as follows. 

An object of $\cont^0_{2n-1}$ is $(Y,\lambda) \in \cont_{2n-1}$ such that $\lambda$ is nondegenerate, 
i.e. for any $\gamma \in P(Y, \lambda)$, $1$ is not an eigenvalue of the linearized return map of $\gamma$
on $(\xi_\lambda)_{\gamma(0)}$. 

For any $(Y_+, \lambda_+), (Y_-, \lambda_-) \in \cont^0_{2n-1}$, the set of morphisms 
$\Hom^0( (Y_+, \lambda_+), (Y_-, \lambda_-))$ 
consists of 
$[(X, \lambda, i_+, i_-)] \in \Hom ( (Y_+, \lambda_+), (Y_-, \lambda_-))$ 
which satisfies either (a) or (b) of the following:
\begin{itemize} 
\item[(a):] 
$i_+( Y_+ \times \R_{\ge 0}) \cap i_-(Y_- \times \R_{\le 0}) = \emptyset$. 
\item[(b):] 
$X = Y_- \times \R$, $\lambda = e^r \lambda_-$, $i_- = \id_{Y_-  \times \R}$ 
and $i_+  = \ph \times \id_\R$ where $\ph: Y_+ \to Y_-$ is a diffeomorphism 
which satisfies $\ph^* \lambda_- = \lambda_+$. 
\end{itemize} 

Then $\cont^0_{2n-1}$ inherits a symmetric monoidal structure from $\cont_{2n-1}$. 

\begin{rem} 
The identity morphisms
and the natural isomorphisms 
$(\alpha \sqcup \beta) \sqcup \gamma \cong \alpha \sqcup (\beta \sqcup \gamma)$
and $\alpha \sqcup \beta \cong \beta \sqcup \alpha$
are of type (b). 
Here $\alpha, \beta, \gamma$ are objects of $\cont^0_{2n-1}$, 
and $\sqcup$ denotes the monoidal product.
\end{rem} 

The next lemma is useful to define action selecting functors in practice. 

\begin{lem}\label{lem_as_functor_ext} 
Let $\Phi: \cont^0_{2n-1} \to \fvect^{\Z/2}_K$ be a 
symmetric strong monoidal functor such that
the following holds for any 
$(Y, \lambda) \in \cont^0_{2n-1}$: 
\begin{itemize} 
\item[(i):] For any $a,b \in \R_{>0}$ with $a<b$ and $[a,b] \cap \spec_+(Y,\lambda)=\emptyset$, $F^a \Phi(Y,\lambda) = F^b \Phi(Y,\lambda)$. 
\item[(ii):] For any $s \in (0,1]$, $\Phi(\mu_{s\lambda, \lambda}): \Phi(Y,\lambda) \to \Phi(Y, s\lambda)$ is an isomorphism of 
vector spaces. 
\item[(iii):]  For any $a \in \R_{>0}$,  $\dim_K F^a \Phi (Y, \lambda) < \infty$. 
\item[(iv):] For any $a, s \in \R_{>0}$, $\dim_K F^a \Phi(Y,\lambda) = \dim_K F^{sa} \Phi(Y, s\lambda)$. 
\end{itemize} 
Then $\Phi$ extends to an action selecting functor $\cont_{2n-1} \to \fvect^{\Z/2}_K$. 
Moreover, such extension is unique up to natural equivalences. 
\end{lem} 
\begin{proof} 
Suppose that $\Phi: \cont^0_{2n-1} \to \fvect^{\Z/2}_K$ satisfies the assumption. 
Then $\Phi$ extends to a symmetric strong monoidal functor 
$\bar{\Phi}: \cont_{2n-1} \to \fvect^{\Z/2}_K$
as follows. 
\begin{itemize}
\item 
For any $(Y, \lambda) \in \cont_{2n-1}$ and $a \in \R_{>0}$, 
let us define 
\[ 
\bar{\Phi}(Y,\lambda):= \varinjlim_h \Phi(Y, e^h\lambda), \qquad 
F^a \bar{\Phi}(Y, \lambda):= \varinjlim_h F^a \Phi (Y, e^h\lambda), 
\] 
where $h$ runs over elements of $C^\infty(Y, \R_{>0})$ such that 
$e^h\lambda$ is nondegenerate. 

\item 
For any $(Y_+, \lambda_+), (Y_-, \lambda_-) \in \cont_{2n-1}$ 
and $[(X, \lambda, i_+, i_-)] \in \Hom( (Y_+, \lambda_+), (Y_-, \lambda_-))$, 
let us define 
\[ 
\bar{\Phi}( [ (X, \lambda, i_+, i_-)])
:= \varinjlim_{(h_+, h_-)}  \Phi ( [ (X, \lambda, i_{h_+}, i_{h_-})]), 
\] 
where $(h_+, h_-)$ runs over elements of $C^\infty(Y_+ , \R_{>0}) \times C^\infty(Y_-, \R_{>0})$ 
such that the following holds: 
\begin{itemize} 
\item Both $e^{h_+}\lambda_+$ and $e^{h_-}\lambda_-$ are nondegenerate contact forms. 
\item Let us define 
\begin{align*} 
i_{h_+}&:  \del Y_+ \times \R \to X; \,   (y,r) \mapsto i_+(y, h_+(y) + r),  \\ 
i_{h_-}&:  \del Y_- \times \R \to X; \, (y,r) \mapsto i_-(y, h_-(y) + r). 
\end{align*} 
Then $i_{h_+}(Y_+  \times \R_{\ge 0}) \cap i_{h_-}(Y_- \times \R_{\le 0}) = \emptyset$. 
\end{itemize} 

\item
For any $(Y_1, \lambda_1), (Y_2,\lambda_2) \in \cont_{2n-1}$, 
an isomorphism 
\[ 
\bar{\Phi}(Y_1, \lambda_1) \otimes \bar{\Phi}(Y_2, \lambda_2) \to \bar{\Phi} ( Y_1 \sqcup Y_2, \lambda_1 \sqcup \lambda_2)
\] 
is defined as 
\[ 
\varinjlim_{(h_1, h_2)}   \Big(  \Phi( Y_1, e^{h_1} \lambda_1) \otimes \Phi(Y_2, e^{h_2} \lambda_2) 
\to \Phi (Y_1 \sqcup Y_2, e^{h_1} \lambda_1 \sqcup e^{h_2} \lambda_2) \Big), 
\] 
where the limit is taken over the set of pairs 
 $(h_1, h_2) \in C^\infty (Y_1, \R_{> 0}) \times C^\infty(Y_2, \R_{> 0})$ 
such that both 
$e^{h_1} \lambda_1$ and $e^{h_2} \lambda_2$ are nondegenerate contact forms. 
\end{itemize} 

To show that $\bar{\Phi}$ is an action selecting functor, 
it is sufficient to check that the following holds for any 
$a,b \in \R_{>0}$ such that $a<b$ and $[a,b] \cap \spec_+(Y,\lambda)=\emptyset$: 
\begin{itemize} 
\item $F^a\bar{\Phi}(Y,\lambda) = F^b \bar{\Phi}(Y,\lambda)$. 
\item Let $c:=a/b$. Then the linear map 
$\bar{\Phi} ( \mu_{c\lambda, \lambda})|_{F^a\bar{\Phi}(Y,\lambda)}: F^a\bar{\Phi}(Y,\lambda) \to F^a \bar{\Phi} (Y, c\lambda)$ is an isomorphism. 
\end{itemize} 
We may assume that $(Y, \lambda) \in \cont^0_{2n-1}$, since 
the general case follows from this case by taking limits. 
Then, the first condition is the same as assumption (i). 
The map in the second condition is injective, since 
$\Phi(\mu_{c\lambda, \lambda}): \Phi(Y, \lambda) \to \Phi(Y, c\lambda)$ is an isomorphism 
by assumption (ii). 
This map is also surjective, since 
\[ 
\dim_K F^a\Phi(Y,\lambda) = \dim_K F^b \Phi(Y,\lambda) = \dim_K F^a \Phi(Y, c\lambda) < \infty 
\] 
by assumptions (i), (iii) and (iv). 

The uniqueness follows from Lemma \ref{lem_uniqueness_of_extension} below. 
\end{proof} 

\begin{lem}\label{lem_uniqueness_of_extension} 
Let $\Phi: \cont_{2n-1} \to \fvect^{\Z/2}_K$ be an action selecting functor. 
For any $(Y, \lambda) \in \cont_{2n-1}$ and $a \in (-\infty, \infty]$, 
\[ 
\varinjlim_h F^a \Phi(Y, e^h\lambda) \to F^a \Phi(Y, \lambda) 
\] 
is an isomorphism, where $h$ runs over elements of $C^\infty(Y, \R_{>0})$ such that $e^h \lambda$ is nondegenerate. 
\end{lem} 
\begin{proof} 
If $a \le 0$ then the assertion is obvious. 
In the case $a >0$, we may assume that $a \notin \spec_+(Y, \lambda) \cup \{\infty\}$ 
because this assumption can be dropped by taking limits. 
It is sufficient to show that $F^a \Phi(Y, e^h \lambda) \to F^a \Phi (Y, \lambda)$ is an isomorphism 
for any $h \in C^\infty(Y, \R_{>0})$ such that $\| h \|_{C^2}$ is sufficiently small. 

Since $a \notin \spec_+(Y,\lambda) \cup \{ \infty\}$, there exists $\ep>0$ (depending on $a$ and $\lambda$) such that, 
if $h \in C^\infty(Y, \R_{>0})$ satisfies $\| h \|_{C^2} < \ep$ then 
\[ 
[a, e^{\max h} a] \cap \spec_+(Y, e^{\max h} \lambda) = [ a, e^{\max h} a] \cap \spec_+(Y, e^h \lambda) = \emptyset. 
\] 
Since $\Phi$ is an action selecting functor, 
\[ 
F^a \Phi (Y, e^{\max h} \lambda) \to F^a \Phi (Y, \lambda), \qquad 
F^a\Phi (Y, e^h \lambda) \to F^a \Phi (Y, e^{h - \max h} \lambda) 
\] 
are isomorphisms. Then $F^a \Phi (Y, e^h \lambda) \to F^a \Phi (Y, \lambda)$ is also an isomorphism. 
\end{proof}

\subsection{Examples of action selecting functors} 

We explain two examples of action selecting functors $\Phi^\ech$ and $\Phi^{\CH}$, 
which are defined respectively by embedded contact homology (ECH) and (full) contact homology. 

{\bf Embedded contact homology (ECH).} 
Our first example is 
$\Phi^{\ech}: \cont_3 \to \fvect_{\Z/2}$
which is defined by ECH. 
By Lemma \ref{lem_as_functor_ext}, 
it is sufficient to define 
$\Phi^{\ech}: \cont^0_3 \to \fvect_{\Z/2}$ 
which satisfies 
the conditions (i)--(iv). 
For any $(Y,\lambda) \in \cont^0_3$ and $a \in \R_{>0}$, we set 
\begin{align*} 
\Phi^{\ech}(Y,\lambda)&:= \ech_*(Y, \lambda: \Z/2), \\ 
F^a \Phi^{\ech}(Y,\lambda)&:= \image(\ech^a_*(Y, \lambda: \Z/2) \to \ech_*(Y, \lambda: \Z/2) ). 
\end{align*} 
The invariants on the RHS are defined in \cite{Hutchings_Taubes_Chord_II} Sections 1.1--1.3. 

For any $(Y_+, \lambda_+), (Y_-, \lambda_-) \in \cont^0_3$ and 
$\mu \in \Hom^0( (Y_+, \lambda_+), (Y_-, \lambda_-))$, 
let us define $\Phi^\ech(\mu)$. 
Take a tuple $(X, \lambda, i_+, i_-)$ which represents $\mu$. 

In case (a), i.e. $i_+(Y_+ \times \R_{\ge 0}) \cap i_-(Y_- \times \R_{\le 0}) = \emptyset$, let 
\begin{equation}\label{eqn_X_0} 
X_0:= X \setminus ( i_+ ( Y_+ \times \R_{>0}) \cup i_-( Y_- \times \R_{<0})).
\end{equation} 
Then $(X_0, \lambda|_{X_0})$ is an exact symplectic cobordism from 
$(Y_+, \lambda_+)$ to $(Y_-, \lambda_-)$. 
By \cite{Hutchings_Taubes_Chord_II} Theorem 1.9, 
one can define a linear map 
$\Phi^\ech(Y_+, \lambda_+) \to \Phi^\ech(Y_-, \lambda_-)$ 
which preserves filtrations. 
This map depends only on $\mu$, 
since it depends only on (the homotopy class of) the cobordism $(X_0, d \lambda|_{X_0})$. 
Then we define $\Phi^{\ech}(\mu)$ to be this map. 

\begin{rem} 
Although \cite{Hutchings_Taubes_Chord_II} Theorem 1.9 assumes that both $Y_+$ and $Y_-$ are connected, 
this assumption can be dropped if one does not require the direct limit axiom. 
See \cite{Hutchings_Taubes_Chord_II} Remark 1.12.
\end{rem} 

In case (b), i.e. $X = Y_- \times \R$, $\lambda = e^r \lambda_-$, $i_-  = \id_{Y_- \times \R}$
and $i_+ = \ph \times \id_\R$ 
where $\ph: Y_+  \to Y_-$ is a diffeomorphism which satisfies $\ph^*\lambda_-  =\lambda_+$, 
one can define a bijection 
\[ 
P(Y_+, \lambda_+) \to P(Y_-, \lambda_-); \quad \gamma \mapsto \ph \circ \gamma. 
\] 
This bijection induces an isomorphism 
$\Phi^\ech(Y_+, \lambda_+) \to \Phi^\ech(Y_-, \lambda_-)$
which preserves filtrations and depends only on $\mu$. 
We define $\Phi^{\ech}(\mu)$ to be this map. 

For any $(Y_1, \lambda_1), (Y_2, \lambda_2) \in \cont^0_3$, there exists an isomorphism 
\begin{equation}\label{eqn_product_ech} 
\ech_*(Y_1 \sqcup Y_2, \lambda_1 \sqcup \lambda_2) \cong \ech_*(Y_1, \lambda_1) \otimes \ech_*(Y_2, \lambda_2), 
\end{equation} 
which follows from a canonical isomorphism between underlying chain complexes 
(see \cite{Hutchings_quantitative} Section 2.2). 
The map (\ref{eqn_product_ech}) 
preserves action filtrations as proved in Step 2 in the proof of \cite{Hutchings_quantitative} Proposition 4.6. 

Let us explain how conditions (i)--(iv) are verified. 
For any $a \in \R_{>0}$, 
$\ech^a_*(Y, \lambda)$ is the homology of 
$(\ecc^a_*(Y, \lambda), \del_J)$, where $J$ is a generic symplectization-admissible almost complex structure 
on $Y \times \R$ (see \cite{Hutchings_Taubes_Chord_II} page 2602). 
From the definition of this chain complex, it is easy to see that 
$\dim \ecc^a_*(Y, \lambda)<\infty$ (which implies (iii)) and 
$\ecc^a_*(Y, \lambda) = \ecc^b_*(Y, \lambda)$ if $[a,b] \cap \spec_+(Y,\lambda) = \emptyset$ (which implies (i)). 
(ii) follows from \cite{CGHR} formula (53). 
(iv) holds since $\ech^a_*(Y,\lambda) \cong \ech^{sa}_*(Y, s\lambda)$
for any $a,s \in\R_{>0}$; see the following remark. 

\begin{rem}\label{rem_scaling_isom} 
For the proof of the above isomorphism, see \cite{Hutchings_Taubes_Chord_II} Section 1.2. 
The point is that, if an almost complex structure $J$ on $Y \times \R$ is symplectization-admissible with respect to $\lambda$, 
then one can define an almost complex structure $J^s$ on $Y \times \R$ which is symplectization-admissible with respect to $s\lambda$, and a self-diffeomorphism $f: Y \times \R \to Y \times \R; \, (y,r) \mapsto (y, sr)$ satisfies $f^* (J^s) = J$. 
\end{rem}

{\bf Contact homology.} 
Let $n \in \Z_{\ge 1}$. 
Our next example is 
$\Phi^{\CH}: \cont_{2n-1} \to \fvect^{\Z/2}_\Q$, 
which is defined by (full) contact homology. 
For constructions of contact homology, see \cite{Bao_Honda}, \cite{Ishikawa}, \cite{Pardon} and references therein. 
The following argument follows \cite{Pardon}. 

By Lemma \ref{lem_as_functor_ext}, 
it is sufficient to define 
$\Phi^{\CH}: \cont^0_{2n-1} \to \fvect^{\Z/2}_\Q$ 
which satisfies 
the conditions (i)--(iv).
For any $(Y,\lambda) \in \cont^0_{2n-1}$, we set 
\begin{align*} 
\Phi^{\CH}(Y,\lambda)&:= \varinjlim_{a \to \infty} \CH_*(Y, \lambda)^{<a},  \\ 
F^a \Phi^{\CH}(Y,\lambda)&:= \image ( \CH_*(Y, \lambda)^{<a} \to \Phi^{\CH}(Y, \lambda)). 
\end{align*} 
The chain complex underlying $\CH_*(Y, \lambda)^{<a}$ is (as a vector space) spanned by 
formal products of (unparameterized) good periodic Reeb orbits of $\lambda$ with total action less than $a$. 
See the third bullet (Action filtration) of \cite{Pardon} Section 1.8. 

For any $(Y_+, \lambda_+), (Y_-, \lambda_-) \in \cont^0_{2n-1}$ and 
$\mu \in \Hom^0( (Y_+, \lambda_+), (Y_-, \lambda_-))$, 
let us define $\Phi^{\CH}(\mu)$. 
Take a tuple $(X, \lambda, i_+, i_-)$ which represents $\mu$. 

In case (a), i.e. $i_+(Y_+ \times \R_{\ge 0}) \cap i_-(Y_- \times \R_{\le 0}) = \emptyset$, let 
us define $X_0$ by (\ref{eqn_X_0}), 
then $(X_0, \lambda|_{X_0})$ is an exact symplectic cobordism from 
$(Y_+, \lambda_+)$ to $(Y_-, \lambda_-)$. 
Then one can define a linear map 
$\Phi^{\CH}(Y_+, \lambda_+) \to \Phi^{\CH}(Y_-, \lambda_-)$ 
which preserves filtrations and depends only on $\mu$ (see \cite{Pardon} Section 1.3). 
We define $\Phi^{\CH}(\mu)$ to be this map. 

In case (b), i.e. $X = Y_- \times \R$, $\lambda = e^r \lambda_-$, $i_-  = \id_{Y_- \times \R}$
and $i_+ = \ph \times \id_\R$ 
where $\ph: Y_+  \to Y_-$ is a diffeomorphism which satisfies $\ph^*\lambda_-  =\lambda_+$, 
one can define a bijection 
\[ 
P(Y_+, \lambda_+) \to P(Y_-, \lambda_-); \quad \gamma \mapsto \ph \circ \gamma. 
\] 
This bijection induces an isomorphism 
$\Phi^{\CH}(Y_+, \lambda_+) \to \Phi^{\CH}(Y_-, \lambda_-)$
which preserves filtrations and depends only on $\mu$. 
We define $\Phi^{\CH}(\mu)$ to be this map. 

For any $(Y_1, \lambda_1), (Y_2, \lambda_2) \in \cont^0_{2n-1}$, 
there exists a canonical isomorphism 
\begin{equation}\label{eqn_prod_ch} 
\Phi^{\CH}(Y_1 \sqcup Y_2, \lambda_1 \sqcup \lambda_2) \cong  \Phi^{\CH}(Y_1, \lambda_1) \otimes \Phi^{\CH}(Y_2, \lambda_2) 
\end{equation}
which follows from a canonical isomorphism between underlying chain complexes (see \cite{Pardon} Proposition 4.36). 
This chain level isomorphism preserves filtrations, 
and the underlying chain complexes satisfy the assumption in Lemma \ref{lem_filtered_chain_complex}
(see the paragraph below), 
thus the map (\ref{eqn_prod_ch}) preserves filtrations. 

Let us explain how conditions (i)--(iv) are verified. 
For any $a \in \R_{>0}$, 
$\CH_*(Y, \lambda)^{<a}$ is the homology of a chain complex  
$\mathrm{CC}_*(Y, \xi_\lambda)^{<a}_{\lambda, J, \theta}$ (which we abbreviate by $\mathrm{CC}^a$), 
where $J$ is an almost complex structure on $\xi_\lambda$ which is compatible with $d\lambda$, 
and $\theta$ is a specification of ``perturbation data'' for relevant moduli spaces (see \cite{Pardon} Section 1.2). 
From the definition of this chain complex, 
it is easy to see that 
$\dim \mathrm{CC}^a < \infty$ for any $a<\infty$ (which implies (iii)), 
and 
$\mathrm{CC}^a = \mathrm{CC}^b$ 
if $[a,b] \cap \spec_+(Y, \lambda) = \emptyset$ (which implies (i)). 
(ii) follows from Lemma 1.2 in \cite{Pardon}. 
(iv) holds since
$\CH_*(Y,\lambda)^{<a} \cong \CH_*(Y, s\lambda)^{<sa}$
for any $a,s \in \R_{>0}$, by the reasoning similar to Remark \ref{rem_scaling_isom}. 

\section{A sufficient criterion for strong closing property} 

After some preparations in Section 4.1, 
we state and prove a sufficient criterion for strong closing property (Theorem \ref{mainthm}) in Section 4.2. 
In Section 4.3, we prove Theorem \ref{thm_SCP_dim3} as an application of Theorem \ref{mainthm}, combined with results in \cite{CGHR}. 

\subsection{Standard contact sphere} 

Let us consider $\R^{2n}$ with coordinates $q_1, \ldots, q_n, p_1, \ldots, p_n$. 
We will use the following notations: 
\begin{align*} 
\lambda_{2n}&:= \sum_{i=1}^n \frac{p_i dq_i - q_i dp_i}{2} \in \Omega^1(\R^{2n}),  \\ 
B^{2n}(r)&:= \bigg\{(q_1, \ldots, q_n, p_1, \ldots, p_n) \in \R^{2n} \biggm{|}  \sum_{i=1}^n q_i^2 + p_i^2 \le r^2  \bigg\},  \\ 
S^{2n-1}&:= \del B^{2n}(1), \qquad \xi_{2n-1}:= \xi_{\lambda_{2n}|_{S^{2n-1}}}. 
\end{align*} 

For any action selecting functor $\Phi: \cont_{2n-1} \to \fvect^{\Z/2}_K$, let 
\[ 
A_\Phi:= \Phi(S^{2n-1}, \xi_{2n-1})^\vee =  \Hom_K  ( \Phi(S^{2n-1}, \xi_{2n-1}) , K). 
\] 
Note that $A_\Phi$ has a natural $\Z/2$-grading. 

The goal of this subsection is to define an algebra structure on $A_\Phi$, 
and an $A_\Phi$-module structure on $\Phi(Y,\xi)$ 
for any $2n-1$-dimensional closed and connected contact manifold $(Y, \xi)$. 
Let us first prove a few elementary lemmas. 

\begin{lem}\label{lem_one_form} 
Let $(Y, \lambda) \in \cont_{2n-1}$ 
and $h \in C^\infty(Y, \R_{\ge 0}) \setminus \{0\}$. 
Let 
\[ 
U_h:= \{ (y,r) \in Y \times \R \mid 0 < r < h(y) \}. 
\] 
Let $\ep \in \R_{>0}$ and $\ph: B^{2n}(1) \to U_h$ be an embedding such that 
$\ph^* d(e^r \lambda) = d(\ep \lambda_{2n})$. 
Then there exists $\eta \in \Omega^1(Y \times \R )$ such that 
$\supp (\eta - e^r \lambda) \subset U_h$, 
$d \eta = d (e^r\lambda)$, and 
$\ph^* \eta = \ep \lambda_{2n}$. 
\end{lem}
\begin{proof} 
There exists $\eta' \in \Omega^1(Y \times \R)$ such that 
$\supp (\eta' - e^r\lambda) \subset U_h$ 
and $\ph^* \eta' = \ep \lambda_{2n}$. 
Then $\supp d(\eta' - e^r\lambda) \subset U_h  \setminus \image \ph$. 
Since 
$H^2_{c, \dR}( U_h  \setminus \image \ph) \to H^2_{c, \dR}(U_h) $ is injective, 
there exists 
$\eta'' \in \Omega^2(Y \times \R)$ such that 
$\supp \eta'' \subset U_h  \setminus \image \ph$ and 
$d\eta'' = d ( \eta' - e^r\lambda)$. 
Then $\eta:= \eta' - \eta''$ satisfies the required conditions. 
\end{proof} 

\begin{lem}\label{lem_balls_in_symplectic_manifolds} 
Let $(X, \omega)$ be a $2n$-dimensional symplectic manifold, 
$k \in \Z_{\ge 1}$, $\ep \in \R_{>0}$, and 
$\ph_1, \ldots, \ph_k, \ph'_1, \ldots, \ph'_k$ be symplectic embeddings from 
$(B^{2n}(1),  d (\ep\lambda_{2n}))$ to $(X, \omega)$ such that 
\[ 
1 \le i < j \le k  
\implies \ph_i(0, \ldots, 0) \ne \ph_j(0, \ldots, 0), \quad \ph'_i(0,\ldots,0) \ne \ph'_j(0,\ldots,0). 
\] 
If $X$ is connected, 
then there exist a compactly-supported diffeomorphism $\psi: X \to X$ and $r \in (0, 1)$ such that 
$\psi^* \omega = \omega$ and 
$\psi \circ \ph_i |_{B^{2n}(r)} = \ph'_i |_{B^{2n}(r)}$ for any $i$. 
\end{lem} 
\begin{proof} 
For each $i$, let 
$x_i:= \varphi_i(0,\ldots,0)$ and 
$x'_i:= \varphi'_i(0,\ldots,0)$. 
Since $x_1,\ldots,x_k$ and $x'_1,\ldots, x'_k$ are distinct and $X$ is connected, 
we may assume that $x_i = x'_i$ for any $i$. 
Moreover, since $\mathrm{Sp}(2n, \R)$ is connected, we may assume that 
$(d\ph_i)_{(0,\ldots,0)} = (d\ph'_i)_{(0,\ldots,0)}$ for any $i$. 
Then, for each $i$, there exists a pair $(U_i, \psi_i)$ 
such that 
\begin{itemize} 
\item $U_i$ is an open neighborhood of $x_i$ 
such that $U_i \subset \image \ph_i$, 
\item 
$\psi_i: U_i \to X$ is an embedding such that 
$\psi_i^*\omega=\omega$ and $\ph'_i|_{\ph_i^{-1}(U_i)}  = \psi_i \circ \ph_i|_{\ph_i^{-1}(U_i)}$, 
\item 
$\psi_i(x_i)=x_i$, $(d\psi_i)_{x_i} = \id_{T_{x_i}X}$. 
\end{itemize} 
Then lemma \ref{lem_balls_in_symplectic_manifolds} follows from 
Lemma \ref{lem_generating_function} which we state below. 
\end{proof} 

\begin{lem}\label{lem_generating_function} 
Let $U$ be an open neighborhood of $(0,\ldots, 0)$ in $\R^{2n}$, 
and $f: U \to \R^{2n}$ be an open embedding such that $f^*(d\lambda_{2n})=d\lambda_{2n}$. 
Suppose that $f(0,\ldots,0)=(0,\ldots,0)$ and
$(df)_{(0,\ldots,0)}=\id_{\R^{2n}}$. 
Then, for any open neighborhood $V$ of $(0,\ldots,0)$ in $\R^{2n}$, 
there exists a diffeomorphism $g$ of $\R^{2n}$
such that $g^*(d\lambda_{2n})=d\lambda_{2n}$, 
$\supp g \subset V$, and 
$f \equiv g$ on some neighborhood  $(0,\ldots,0)$. 
\end{lem}
\begin{proof} 
Let us abbreviate $q_1,\ldots, q_n$ by $q$, $p_1,\ldots, p_n$ by $p$, 
and $(0,\ldots, 0) \in \R^n$ by $0$. 
Moreover, let 
\begin{align*} 
\Gamma&:= \{ (q,p,Q,P) \in \R^{2n} \times \R^{2n} \mid f(q,p)=(Q,P)\}, \\ 
\pi&: \R^{2n} \times \R^{2n} \to \R^n \times \R^n; \, (q,p,Q,P) \mapsto (Q,p). 
\end{align*} 
By $f(0,0)=(0,0)$ and $(df)_{(0,0)} = \id_{\R^{2n}}$, 
there exists a contractible open neighborhood $U'$ of 
$(0,0) \in \R^n \times \R^n$ such that 
$\pi|_{\tilde{U'}}: \tilde{U'} \to U'$ is a diffeomorphism, where $\tilde{U'}:= \pi^{-1}(U') \cap \Gamma$. 
Since $\sigma:= qdp+PdQ \in \Omega^1(\tilde{U'})$ is exact, 
there exists $F \in C^\infty(U')$ such that 
$F(0,0)=0$ 
and $\sigma = d ( (\pi|_{\tilde{U'}})^*F)$. 
Then 
\[ 
f(q,p) = (Q,P) \iff  q = \del_pF(Q,p), \quad P = \del_Q F(Q,p)
\] 
for any $(q,p,Q,P) \in \pi^{-1}(U')$. 
Such a function $F$ is called a generating function; see \cite{Hofer_Zehnder} Appendix A.1. 

Let $G(Q,p):= F(Q,p) - Q \cdot p$. 
Then the $m$-th derivative of $G$ at $(0,0)$ vanishes for $m=0, 1, 2$. 
Take an open neighborhood $V'$ of $0$ in $\R^n$ such that 
$V' \times V' \subset V$. 
For any $\ep \in \R_{>0}$, 
there exists $G_\ep \in C^\infty(\R^n \times \R^n)$ such that 
$\supp G_\ep \subset V' \times V'$, $\|G_\ep \|_{C^2} < \ep$ 
and $G_\ep \equiv G$ on some neighborhood of $(0, 0)$. 
Let $F_\ep(Q,p):=G_\ep(Q,p)+Q \cdot p$. 
If $\ep$ is sufficiently close to $0$, 
there exists a diffeomorphism $g$ of $\R^{2n}$ such that 
\[ 
g(q,p) = (Q,P) \iff q = \del_p F_\ep(Q,p), \quad P = \del_Q F_\ep(Q,p). 
\] 
Then $g$ satisfies the requirements of this lemma. 
\end{proof}

Let $(Y, \xi)$ be a closed and connected contact manifold of dimension $2n-1$. 
Take $\lambda \in \Lambda(Y, \xi)$,
$h \in C^\infty(Y, \R_{\ge 0}) \setminus \{0\}$, 
an embedding $\ph: B^{2n}(1) \to U_h$
and $\eta \in \Omega^1(Y \times \R)$ 
such that 
$\supp (\eta - e^r \lambda) \subset U_h$, 
$d\eta = d(e^r \lambda)$, 
and $\ph^*\eta = \ep \lambda_{2n}$.
Define $Z \in \mca{X}(Y \times \R)$ by 
$i_Z (d\eta)= \eta$. 

Let $X:= Y \times \R \setminus \{ \ph(0, \ldots, 0) \}$. 
Define an embedding $i_+: Y \times \R \to X$ by 
\[ 
i_+(y, 0) = (y, h(y)) \quad (\forall y \in Y),  \quad 
(i_+)_* ( \del_r) = Z \circ i_+. 
\] 
Define an embedding $i_-: (S^{2n-1} \sqcup Y) \times \R \to X$ by 
\[ 
i_-(z,0) = \ph(z) \quad (\forall z \in S^{2n-1}), \quad 
i_-(y, 0) = (y, 0) \quad  (\forall y \in Y), \quad 
(i_-)_*( \del_r) = Z \circ i_-. 
\] 
Then $[(X, \eta|_X, i_+, i_-)] \in \Hom( (Y, e^h \lambda),  (S^{2n-1}, \ep \lambda_{2n}) \sqcup (Y, \lambda) )$. 

\begin{center}
\includegraphics[width=5cm]{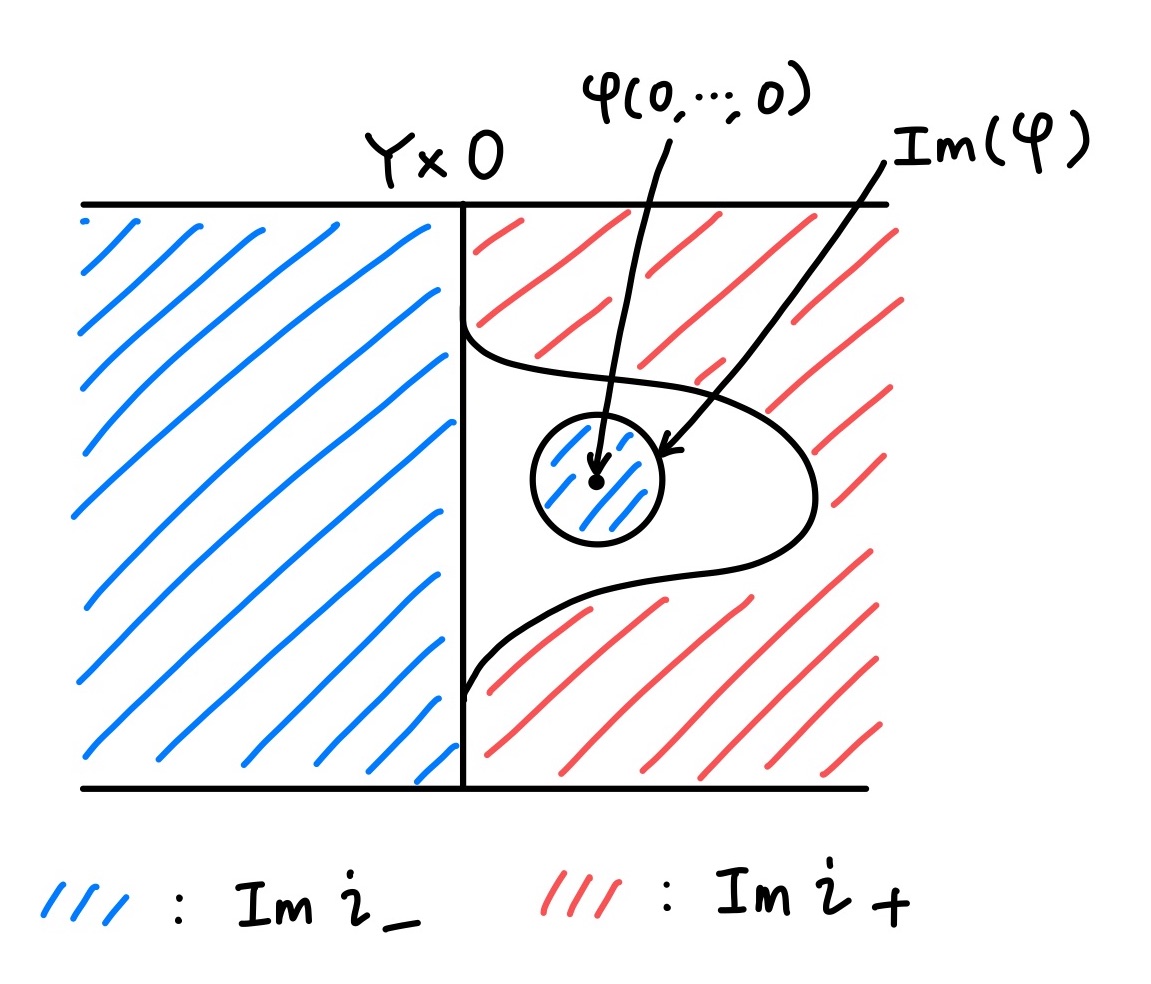}
\end{center} 

Since $\Phi$ is a strong monoidal functor, 
we can define a linear map 
\begin{equation}\label{eqn_Y_Y_S}
\Phi (Y, \xi) \to \Phi (S^{2n-1}, \xi_{2n-1}) \otimes \Phi(Y, \xi). 
\end{equation}
This map (\ref{eqn_Y_Y_S}) does not depend on the choice of $\eta$, $h$, $\ep$ and $\ph$. 
Indeed, (\ref{eqn_Y_Y_S}) does not depend on $\eta$ by the definition of morphisms in $\cont_{2n-1}$. 
Moreover, since $Y$ is connected, 
applying Lemma \ref{lem_balls_in_symplectic_manifolds} 
for  $(X, \omega) = (Y \times \R, d (e^r\lambda))$ and $k=1$, 
we conclude that (\ref{eqn_Y_Y_S}) does not depend on $h$, $\ep$ and $\ph$. 

Putting $(Y, \xi) = (S^{2n-1}, \xi_{2n-1})$ into (\ref{eqn_Y_Y_S}), we obtain 
\begin{equation}\label{eqn_S_S_S}
\Delta:  \Phi (S^{2n-1},  \xi_{2n-1}) \to \Phi (S^{2n-1} , \xi_{2n-1})^{\otimes 2}. 
\end{equation} 
Let us consider 
\[ 
\langle \, , \, \rangle:   A_\Phi^{\otimes 2} \otimes \Phi(S^{2n-1}, \xi_{2n-1})^{\otimes 2} \to K; \quad 
(\alpha \otimes \beta) \otimes (a \otimes b) \mapsto  (-1)^{|a||\beta|} \alpha(a)  \beta(b), 
\] 
and define $\nabla: A_\Phi^{\otimes 2} \to A_\Phi$ by 
$\nabla ( \alpha \otimes \beta) (a) := \langle \alpha \otimes \beta,   \Delta(a) \rangle$. 

\begin{prop}\label{prop_A_phi} 
$(A_\Phi, \nabla)$ is a graded-commutative and associative algebra. 
$A_\Phi \otimes \Phi(Y, \xi) \to \Phi(Y,\xi)$, obtained from (\ref{eqn_Y_Y_S}), 
makes $\Phi(Y, \xi)$ an  $A_\Phi$-module. 
Moreover, for any $\sigma \in \Phi(Y,\xi)$, $\alpha \in A_\Phi$ and $\lambda \in \Lambda(Y,\xi)$, 
there holds $c^\Phi_{\alpha\sigma}(\lambda) \le c^\Phi_\sigma(\lambda)$. 
\end{prop} 
\begin{proof} 
To prove graded commutativity (resp. associativity) of $A_\Phi$, 
apply Lemma \ref{lem_balls_in_symplectic_manifolds} 
for $(X, \omega)= (\R^{2n}, d \lambda_{2n})$ and $k=2$ (resp. $k=3$). 
To prove that $\Phi(Y, \xi)$ is an $A_\Phi$-module, 
take $\lambda \in \Lambda(Y, \xi)$ and apply Lemma \ref{lem_balls_in_symplectic_manifolds} 
for $(X, \omega) = ( Y \times \R , d (e^r \lambda))$ and $k=2$. 

To prove the last assertion, take $a \in \R_{>0}$ arbitrarily, 
and take $\ep \in \R_{>0}$ and an embedding $\ph: B^{2n}(1)  \to Y \times (0, a)$ 
so that $\ph^* d(e^r\lambda) = d(\ep \lambda_{2n})$. 
Recall that for any vector spaces $V$, $W$ and $\psi \in V^{\vee}$, 
we defined $i_\psi: V \otimes W \to W$ by 
$i_\psi(v \otimes w):= \psi(v) w$. 
Then we have the following commutative diagram: 
\[
\xymatrix{ 
\Phi(Y, e^a\lambda) \ar[d]_-{I_{e^a\lambda}}\ar[r]  & \Phi(S^{2n-1}, \ep \lambda_{2n})  \otimes \Phi(Y, \lambda) \ar[d]_-{I_{\ep\lambda_{2n}} \otimes I_\lambda}\ar[r]^-{i_{\alpha \circ I_{\ep \lambda_{2n}}} }& \Phi(Y, \lambda)\ar[d]_-{I_\lambda} \\ 
\Phi(Y, \xi) \ar[r]&\Phi(S^{2n-1}, \xi_{2n-1})  \otimes \Phi(Y, \xi) \ar[r]_-{i_\alpha} & \Phi(Y, \xi). }
\] 
The composition of the  lower maps sends $\sigma$ to $\alpha \sigma$, 
and the upper maps preserve filtrations. 
Thus $c^\Phi_{\alpha\sigma}(\lambda) \le c^\Phi_\sigma(e^a\lambda) = e^a c^\Phi_\sigma(\lambda)$. 
Since $a$ can be arbitrarily close to $0$, we obtain $c^\Phi_{\alpha\sigma}(\lambda) \le c^\Phi_\sigma(\lambda)$. 
\end{proof} 

\subsection{A sufficient criterion for strong closing property} 

Let us take $\lambda_{2n}$ as the ``standard'' contact form on $(S^{2n-1}, \xi_{2n-1})$, 
and consider the canonical isomorphism 
$I_{\lambda_{2n}}: \Phi (S^{2n-1}, \lambda_{2n}) \to \Phi( S^{2n-1}, \xi_{2n-1})$. 
For any $\alpha \in A_\Phi \setminus \{0\}$, 
let 
\begin{equation}\label{eqn_alpha} 
|\alpha|:= |\alpha \circ I_{\lambda_{2n}} | = 
\sup  \, \{ a \mid \alpha \circ  I_{\lambda_{2n}} |_{ F^a\Phi (S^{2n-1}, \lambda_{2n})} = 0 \}. 
\end{equation}
Let $A^+_\Phi:= \{ \alpha \in A_\Phi \setminus \{0\} \mid |\alpha|>0\}$. 

\begin{thm}\label{mainthm} 
Let $(Y, \xi )$ be a closed and connected contact manifold of dimension $2n-1$, 
and $\lambda \in \Lambda(Y,\xi)$. 
If there exists an action selecting functor $\Phi$ such that 
\begin{equation}\label{eqn_criterion} 
\inf_{ \substack{ \sigma \in \Phi(Y, \xi) \setminus \{0\} \\ \alpha \in A^+_\Phi}}
\frac{ c^\Phi_\sigma(\lambda) - c^\Phi_{\alpha \sigma}(\lambda)}{|\alpha|} = 0, 
\end{equation} 
then $\lambda$ satisfies positive and negative strong closing property. 
\end{thm} 

\begin{rem}\label{rem_mainthm} 
If there exists 
$\alpha \in A^+_\Phi$ such that 
$\inf_{\sigma \in \Phi(Y, \xi) \setminus \{0\}} c_\sigma^\Phi(\lambda) - c_{\alpha \sigma}^\Phi(\lambda)=0$, 
then (\ref{eqn_criterion}) obviously holds. 
\end{rem} 
\begin{proof}
We prove that $\lambda$ satisfies positive strong closing property. 
By Lemma \ref{lem_local_sensitivity}, 
it is sufficient to show the following: 
for any $h \in C^\infty(Y, \R_{\ge 0}) \setminus \{0\}$, 
there exists $\sigma \in \Phi(Y,\xi)\setminus \{0\}$
such that $c^\Phi_\sigma(e^h\lambda) > c^\Phi_\sigma(\lambda)$. 

Suppose that this does not hold, i.e. 
there exists $h \in C^\infty(Y, \R_{\ge 0}) \setminus \{0\}$ such that 
$c^\Phi_\sigma(e^h\lambda)=c^\Phi_\sigma(\lambda)$ for any $\sigma \in \Phi(Y, \xi) \setminus \{0\}$. 
Take $\ep \in \R_{>0}$ such that there exists an embedding 
$\ph: B^{2n}(1)\to  \{(y,r) \in Y \times \R \mid 0 < r < h(y) \}$
which satisfies 
$\ph^*d(e^r\lambda) = d (\ep \lambda_{2n})$. 
For any $\alpha \in A_\Phi$, consider the following commutative diagram: 
\[
\xymatrix{ 
\Phi(Y, e^h\lambda) \ar[d]_-{I_{e^h\lambda}}\ar[r]  & \Phi(S^{2n-1}, \ep \lambda_{2n}) \otimes \Phi(Y, \lambda)  \ar[d]_-{I_{\ep\lambda_{2n}} \otimes I_\lambda}\ar[r]^-{i_{\alpha \circ I_{\ep \lambda_{2n}}}}& \Phi(Y, \lambda)\ar[d]_-{I_\lambda} \\ 
\Phi(Y, \xi) \ar[r]& \Phi(S^{2n-1}, \xi_{2n-1}) \otimes \Phi (Y, \xi)  \ar[r]_-{i_\alpha} & \Phi(Y, \xi). }
\] 
For any $\sigma \in \Phi(Y, \xi)$, 
the composition of the lower maps sends $\sigma$ to $\alpha \sigma$. 
Thus we have 
\[ 
c^\Phi_\sigma(e^h\lambda) - c^\Phi_{\alpha \sigma}(\lambda) \ge | \alpha \circ I_{\ep \lambda_{2n}} | =
\ep |\alpha \circ I_{\lambda_{2n}}| = \ep |\alpha|,
\] 
where the inequality follows from Lemma \ref{lem_level_decreasing}, 
the first equality follows from Remark \ref{rem_scaling_invariance}, 
and the second equality follows from (\ref{eqn_alpha}). 

If $c^\Phi_\sigma(e^h\lambda)=c^\Phi_\sigma(\lambda)$ for any $\sigma \in \Phi (Y, \xi) \setminus \{0\}$, 
then $c^\Phi_\sigma(\lambda) - c^\Phi_{\alpha\sigma}(\lambda) \ge \ep |\alpha|$
for any $\sigma \in \Phi(Y,\xi) \setminus \{0\}$ and $\alpha \in A_\Phi$, 
which contradicts the assumption. 

To prove that $\lambda$ satisfies negative strong closing property, 
let us assume that there exists $h \in C^\infty(Y, \R_{\le 0}) \setminus \{0\}$ 
such that $c^\Phi_\sigma(e^h\lambda) = c^\Phi_\sigma(\lambda)$ for any $\sigma \in \Phi(Y, \xi) \setminus \{0\}$. 
Take $\ep \in \R_{>0}$ and an embedding 
$\ph: B^{2n}(1) \to \{ (y,r) \in Y \times \R \mid h(y) < r < 0 \}$
such that $\ph^*d(e^r\lambda) = d(\ep \lambda_{2n})$. 
Then we get a contradiction by arguments similar to the above. 
\end{proof} 

\begin{rem} 
It seems that the mechanism in the above proof is similar to the 
``relation between spectral gaps and creation of periodic orbits'' in \cite{Edtmair_Hutchings} Section 6. 
\cite{Edtmair_Hutchings} also obtains a quantitative estimate of periods (\cite{Edtmair_Hutchings} Proposition 6.2), 
which leads to quantitative closing lemmas for area-preserving diffeomorphisms (see \cite{Edtmair_Hutchings} Section 1.2). 
\end{rem}

\subsection{Proof of Theorem \ref{thm_SCP_dim3}}
Let us recall some facts about ECH from \cite{CGHR}. 
Firstly, $\ech(S^3, \xi_3)$ has a basis $(\zeta_k)_{k \ge 0}$ 
which is characterized by $\zeta_0 = [\emptyset]$ and $U \zeta_{k+1} = \zeta_k \, (\forall k \ge 0)$, 
where $U$ denotes the $U$-map. 
Moreover $|\zeta_k|>0  \iff k>0$. 
Let us define $\alpha \in A^+_{\Phi^\ech}$ by $\alpha(\zeta_k):= \begin{cases} 1 &(k = 1),  \\ 0 &(k \ne 1). \end{cases}$
Secondly, for any closed and connected contact three-manifold $(Y, \xi)$, 
there exists a sequence $(\sigma_k)_{k \ge 0}$ in $\ech(Y, \xi) \setminus \{0\}$ 
which satisfies the following
(we abbreviate $c^{\Phi^\ech}$ by $c$): 
\begin{itemize} 
\item[(i):] $\ech(Y,\xi) \to \ech(S^3, \xi_3) \otimes \ech(Y,\xi)$ 
sends  $\sigma_k$ to $\sum_{j=0}^k  \zeta_j \otimes \sigma_{k-j}$ for any $k$. 
\item[(ii):]  For any $\lambda \in \Lambda(Y, \xi)$, there holds $c_{\sigma_k} (\lambda) = O(\sqrt{k})$. 
\end{itemize} 
\begin{rem} 
Such a sequence $(\sigma_k)_k$ can be constructed as follows. 
Take $\Gamma \in H_1(Y:\Z)$ such that 
$c_1(\xi) + 2 \mathrm{PD}(\Gamma) \in H^2(Y: \Z)$ is a torsion element, 
thus  $\ech(Y, \xi, \Gamma)$ has a relative $\Z$-grading. 
Then take $(\sigma_k)_k$ so that 
each $\sigma_k$ is a nonzero homogeneous element of $\ech(Y, \xi, \Gamma)$, 
and $U\sigma_k = \begin{cases} \sigma_{k-1} &(k \ge 1),  \\ 0 &(k=0). \end{cases}$ 
Then (i) follows from \cite{CGHR} Lemma 3.2, 
and (ii) follows from \cite{CGHR} Proposition 2.1. 
\end{rem} 

Let $(Y, \lambda) \in \cont_3$ such that $Y$ is connected. 
Take $(\sigma_k)_k$ for $(Y, \xi_\lambda)$. Then 
\[  
\inf_k  (c_{\sigma_{k+1}}(\lambda) - c_{\alpha \sigma_{k+1}}(\lambda))
= 
\inf_k  (c_{\sigma_{k+1}}(\lambda) - c_{\sigma_k}(\lambda))   \le 0, 
\] 
where the equality holds since (i) implies $\alpha \sigma_{k+1}=\sigma_k$ for any $k$, 
and the inequality follows from (ii). 
By Theorem \ref{mainthm} and Remark \ref{rem_mainthm}, 
$\lambda$ satisfies strong closing property. 
\qed

\begin{rem}\label{rem_rephrasing} 
The above argument is similar to the proof of the $C^\infty$ closing lemma for three-dimensional Reeb flows 
in real-analytic setting (\cite{Irie_remark} Theorem 4.1). 
It also looks similar to the proofs of quantitative closing lemmas for area-preserving 
diffeomorphisms (see \cite{Edtmair_Hutchings} Section 1.2). 
\end{rem}

\section{Boundaries of symplectic ellipsoids} 

For any $n \in \Z_{\ge 1}$, let 
$A_n:= \{ (a_1, \ldots, a_n) \in \R^n \mid 0 < a_1 \le \cdots \le a_n \}$. 
For any $a = (a_1, \ldots, a_n) \in A_n$, let 
\[ 
E_a:= \bigg\{ (q_1, \ldots, q_n, p_1, \ldots, p_n) \in \R^{2n} \biggm{|}  \sum_{i=1}^n \frac{\pi (q_i^2 + p_i^2)}{a_i}  \le 1 \bigg\}. 
\] 
Recall that $\lambda_{2n} = \sum_{i=1}^n \frac{p_i dq_i - q_i dp_i}{2}$. 
Then $\lambda_{2n}|_{\del E_a}$ is a contact form on $\del E_a$ with the Reeb flow 
$R_{\lambda_{2n}|_{\del E_a}} = \sum_{i=1}^n   \frac{2\pi}{a_i} \cdot ( p_i \del_{q_i} - q_i \del_{p_i})$. 

As a potential application of Theorem \ref{mainthm}, we propose the following conjecture. 

\begin{conj}\label{conj_ellipse} 
For any $n \in \Z_{\ge 1}$ and $a \in A_n$, 
$\lambda_{2n}|_{\del E_a}$ satisfies strong closing property. 
\end{conj} 

A recent paper \cite{CDPT} proposed a proof of Conjecture \ref{conj_ellipse}
(see also \cite{CS} which proved a strong closing property for ``$\gamma$-rigid'' Hamiltonian diffeomorphisms and deduced Conjecture \ref{conj_ellipse} from this result).  
The strategy of \cite{CDPT} is to verify a condition called ``vanishing of the contact homology spectral gap'', 
which is a version of (\ref{eqn_criterion}) with $\Phi = \Phi^{\CH}$. 
To prove this condition for all ellipsoids, 
the first step is to prove a stronger (effective) version of the condition for the periodic case (i.e. the case where all Reeb orbits are periodic), 
and the second step is to prove the condition for the general case by an approximation argument using Dirichlet's approximation theorem. 

In an old version of the present paper \cite{Irie_SCP}, the author reduced Conjecture \ref{conj_ellipse} 
to a numerical conjecture (\cite{Irie_SCP} Conjecture 5.3), 
which is based on the guess that only cylinders will appear in computing contact homology cobordism maps between (boundaries of) ellipsoids. 
The author now thinks that this guess is too naive and computing these cobordism maps is a highly nontrivial problem, 
even for the case $n=2$. 

\appendix
\section{Proof of Lemma \ref{lem_filtered_chain_complex}}

The proof given in this appendix is a straightforward generalization of
the argument in \cite{Hutchings_quantitative} Section 5 in a purely algebraic setting. 

\begin{lem}\label{lem_action_minimizing} 
Let $V$ be a filtered vector space such that
$\dim V>0$ and 
$\dim F^aV < \infty$ for any $a \in \R$. 
\begin{enumerate} 
\item[(i):] There exists a basis $(v_k)_{0 \le k < \dim V}$ of $V$ such that 
\[ 
|v_k| = \inf \{ a  \mid \dim F^a V >  k \} 
\] 
for any $k$. We call such a basis action-minimizing. 
\item[(ii):] 
Let $(v_k)_k$ be an action-minimizing basis of $V$. 
Then, the following holds: 
\begin{itemize} 
\item[(a):]  $(|v_k|)_k$ is nondecreasing, i.e. $|v_{k-1}| \le |v_k|$ for any $k \ge 1$. 
\item[(b):]  For any $a \in \R$, $F^a V = \spann \{ v_k \mid |v_k| < a\}$. 
\item[(c):]  For any $a \in \R$, $\{ v_k \mid |v_k| \ge a\}$ is linearly independent in $V/F^aV$. 
\end{itemize} 
\end{enumerate} 
\end{lem} 
\begin{proof} 
(i): 
Since $\dim F^aV<\infty$ for any $a$,
there exists a strictly increasing sequence $(a_i)_{i\ge 1}$ in $\R_{\ge 0}$ such that 
$\spec (V) = \{ a_i \mid i \ge 1\}$. 
Let $d_i:= \dim F^{a_i}V$ for any $i \ge 1$. 
Then $d_1=0$, and the sequence $(d_i)_{i \ge 1}$ is strictly increasing. 
Now we claim that a basis $(v_k)_k$ of $V$ is action-minimizing if and only if the following holds:
\begin{quote}
($\star$): $F^{a_i} V= \spann  \{ v_k \mid 0 \le k < d_i \}$ for any $i \ge 1$. 
\end{quote} 
It is easy to see that there exists a basis $(v_k)_k$ 
satisfying ($\star$). 
Thus this equivalence shows that there exists an action-minimizing basis. 

Suppose that ($\star$) holds. 
For any $0 \le k < \dim V$, 
take maximal $i$ so that $d_i \le k$. 
Then $v_k \in F^a V \iff a>a_i$, thus $|v_k|=a_i$. 
On the other hand $\dim F^a V > k \iff a > a_i$, thus 
$\inf\{ a \mid \dim F^a V  > k \}=a_i=|v_k|$, 
hence the basis $(v_k)_k$ is action-minimizing. 

On the other hand, suppose that a basis $(v_k)_k$ is action-minimizing. 
To verify ($\star$) we may assume $i \ge 2$, since $F^{a_1}V=0$ and $d_1=0$. 
For any $0 \le k < d_i$, we have 
\[
|v_k| = \inf \{ a \mid \dim F^a V  > k \} \le a_{i-1}, 
\] 
where the equality holds since $(v_k)_k$ is action-minimizing, 
and the inequality holds since if $a>a_{i-1}$ then 
$\dim F^a V \ge \dim F^{a_i} V  = d_i > k$. 
Hence $v_k \in F^{a_i}V$ since $a_i> a_{i-1}$. 
Now we have shown that 
$F^{a_i}V \supset \spann \{ v_k \mid 0 \le k < d_i \}$. 
Since both vector spaces have dimension $d_i$, ($\star$) holds. 
This completes the proof. 

(ii): 
(a) follows from the definition of action-minimizing basis. 
(b) follows from ($\star$). 
(c) follows from (b). 
\end{proof} 

The next lemma is the same as \cite{Hutchings_quantitative} Lemma 5.4. 

\begin{lem}\label{lem_cocycle} 
Let $(C, \del)$ be a chain complex, and let $C' \subset C$ be a subcomplex. 
Suppose $\alpha_1, \ldots, \alpha_n \in H_*(C)$ are linearly independent in 
$H_*(C)/H_*(C')$, and let $y_1, \ldots, y_n \in K$. Then there exists a cocycle $\zeta \in \Hom_K(C, K)$ 
which annihilates $C'$ and satisfies $H_*(\zeta)(\alpha_i)=y_i$ for each $i$. 
\end{lem} 

Let us now prove Lemma \ref{lem_filtered_chain_complex}. 
If $\dim H_*(V)=0$ or $\dim H_*(W)=0$ then this lemma is easy to check, 
thus we may assume that both $\dim H_*(V)$ and $\dim H_*(W)$ are positive. 
Let $(v_i)_i$ be an action-minimizing basis of $H_*(V)$, 
and let $(w_j)_j$ be an action-minimizing basis of $H_*(W)$. 
To prove Lemma \ref{lem_filtered_chain_complex}, 
it is sufficient to show 
\begin{equation}\label{eqn_FaH} 
F^a H_*(V \otimes W) = \spann \{ v_i \otimes w_j \mid |v_i| + |w_j|< a\} 
\end{equation} 
for any $a \in \R$. 
This is because there holds 
\[ 
F^b H_*(V) = \spann \{ v_i \mid |v_i|< b\}, \qquad
F^c H_*(W) = \spann \{ w_j \mid |w_j|<c\}
\] 
for any $b, c \in \R$ by Lemma \ref{lem_action_minimizing} (ii)-(b). 

For any $i$, $j$ such that $|v_i|+|w_j|<a$, 
there exists $\ep>0$ such that $|v_i|+|w_j|+2\ep \le a$. 
Then $v_i \in F^{|v_i|+\ep} H_*(V)$, 
thus there exists a cycle $\tilde{v}_i \in F^{|v_i|+\ep}V$ such that $[\tilde{v}_i]=v_i$. 
Similarly, there exists a cycle $\tilde{w}_j \in F^{|w_j|+\ep}W$ such that $[\tilde{w}_j] = w_j$. 
Then $\tilde{v}_i \otimes \tilde{w}_j \in F^a(V \otimes W)$ is a cycle which represents $v_i \otimes w_j$,
thus $v_i \otimes w_j \in F^a H_*(V \otimes W)$. 
Hence the RHS of (\ref{eqn_FaH}) is a subspace of $F^a H_*(V \otimes W)$. 

To prove that $F^a H_*(V \otimes W)$ is a subspace of the RHS of (\ref{eqn_FaH}), 
it is sufficient to show that,  
for any $u = \sum_{i,j} t_{ij}  \cdot v_i \otimes w_j \in F^a H_*(V \otimes W) \setminus \{0\}$, 
\[ 
a':=  \max \{ |v_i|+|w_j| \mid t_{ij} \ne 0 \} < a. 
\] 

Suppose that $a' \ge a$, 
and take $(i_0, j_0)$ so that 
$t_{i_0 j_0} \ne 0$ and 
$|v_{i_0}| + |w_{j_0}|=a'$. 
Also, take positive integers $I$ and $J$ so that 
$t_{ij} \ne 0 \implies i \le I, j \le J$. 
By Lemma \ref{lem_action_minimizing} (ii)-(c), 
$\{ v_i \mid  |v_i| \ge |v_{i_0}|\}$ 
is linearly independent in $H_*(V)/F^{|v_{i_0}|}H_*(V)$. 
Thus by Lemma \ref{lem_cocycle}, 
there exists a cocycle $\ph \in \Hom_K(V, K)$
which annihilates $F^{|v_{i_0}|}V$
and satisfies $H_*(\ph)(v_i) = \begin{cases} 1 &(i = i_0) \\ 0 &(i \ne i_0) \end{cases}$ for any $i \le I$. 
Similarly, there exists a cocycle $\psi \in \Hom_K(W, K)$ 
which annihilates $F^{|w_{j_0}|}W$
and satisfies $H_*(\psi)(w_j) = \begin{cases} 1 &(j=j_0) \\ 0 &(j \ne j_0) \end{cases}$ for any $j \le J$. 
The cocycle $\ph \otimes \psi \in \Hom_K(V \otimes W, K)$ 
annihilates $F^a(V \otimes W)$ 
since $|v_{i_0}| + |w_{j_0}| = a' \ge a$. 
Thus $H_*(\ph \otimes \psi)$ annihilates $F^a H_*(V \otimes W)$. 
On the other hand
\[ 
H_*(\ph \otimes \psi) (u)  = \sum_{i,j} t_{ij}  \cdot  H_*(\ph)(v_i) H_*(\psi)(w_j) = t_{i_0 j_0} \ne 0, 
\] 
which contradicts $u \in F^a H_*(V \otimes W)$. 
Thus we have shown (\ref{eqn_FaH}), which completes the proof of Lemma \ref{lem_filtered_chain_complex}. 
\qed

\end{document}